\RequirePackage{fixltx2e}
\documentclass[oneside,english]{amsart}
\usepackage[T1]{fontenc}
\usepackage[latin9]{inputenc}
\usepackage[a4paper]{geometry}
\usepackage{babel}
\usepackage{prettyref}
\usepackage{enumitem}
\usepackage{amsthm}
\usepackage{amssymb}
\usepackage{mathdots}
\usepackage[numbers]{natbib}
\usepackage[unicode=true]
 {hyperref}

\makeatletter

\providecommand{\tabularnewline}{\\}

\numberwithin{equation}{section}
\numberwithin{figure}{section}
\numberwithin{table}{section}
\theoremstyle{plain}
\newtheorem{thm}{\protect\theoremname}[section]
\theoremstyle{plain}
\newtheorem{fact}[thm]{\protect\factname}
\theoremstyle{definition}
\newtheorem{problem}[thm]{\protect\problemname}
\theoremstyle{remark}
\newtheorem{rem}[thm]{\protect\remarkname}
\theoremstyle{plain}
\newtheorem{cor}[thm]{\protect\corollaryname}
\newlist{casenv}{enumerate}{4}
\setlist[casenv]{leftmargin=*,align=left,widest={iiii}}
\setlist[casenv,1]{label={{\itshape\ \casename} \arabic*.},ref=\arabic*}
\setlist[casenv,2]{label={{\itshape\ \casename} \roman*.},ref=\roman*}
\setlist[casenv,3]{label={{\itshape\ \casename\ \alph*.}},ref=\alph*}
\setlist[casenv,4]{label={{\itshape\ \casename} \arabic*.},ref=\arabic*}

\usepackage[math]{iwona}
\usepackage[T1]{fontenc}
\usepackage{braket}

\usepackage{proof}
\newcommand{\amp}{&}

\mathchardef\mhyphen="2D

\newcommand{\INT}{\mathbf{INT}}
\newcommand{\CL}{\mathbf{CL}}
\newcommand{\GD}{\mathbf{GD}}

\newcommand{\LJ}{\mathbf{LJ}}
\newcommand{\NJ}{\mathbf{NJ}}

\newcommand{\HLK}{\mathbf{HLK}}
\newcommand{\HLJ}{\mathbf{HLJ}}

\newcommand{\ACD}{\mathbf{CD}}
\newcommand{\LIN}{\mathbf{LIN}}
\newcommand{\WLEM}{\mathbf{WLEM}}
\newcommand{\GDM}{\mathbf{GDM}}

\newcommand{\rs}{\mathrm{rs}}
\newcommand{\ls}{\mathrm{ls}}
\newcommand{\com}{\mathrm{com}}

\newcommand{\To}{\mathbin{\Rightarrow}}

\newrefformat{fact}{Fact \ref{#1}}
\newrefformat{cor}{Corollary \ref{#1}}
\newrefformat{tab}{Table \ref{#1}}

\makeatother

\providecommand{\casename}{Case}
\providecommand{\corollaryname}{Corollary}
\providecommand{\factname}{Fact}
\providecommand{\problemname}{Problem}
\providecommand{\remarkname}{Remark}
\providecommand{\theoremname}{Theorem}

\begin{document}
\title{Hypersequent Calculi for Intermediate Predicate Logics}
\thanks{This work was supported by the Research Institute for Mathematical
Sciences, an International Joint Usage/Research Center located in
Kyoto University.}
\author{Takuma Imamura}
\address{Research Institute for Mathematical Sciences\\
Kyoto University\\
Kitashirakawa Oiwake-cho, Sakyo-ku, Kyoto 606-8502, JAPAN}
\email{\href{mailto:timamura@kurims.kyoto-u.ac.jp}{timamura@kurims.kyoto-u.ac.jp}}
\thanks{The first author was supported by the Morikazu Ishihara (Shikata)
Research Encouragement Fund and by JST ERATO HASUO Metamathematics
for Systems Design Project (No. JPMJER1603).}
\author{Shuya Matsumoto}
\address{Department of Mathematics, Faculty of Science and Technology\\
Keio University\\
3-14-1, Hiyoshi, Kohoku-ku, Yokohama-shi, Kanagawa 223-8522, JAPAN}
\email{\href{mailto:syuyamatsumoto@keio.jp}{syuyamatsumoto@keio.jp}}
\author{Shin Quawai}
\address{Independent Researcher\\
Room 602, Wisteria Shirakawa, 15, Ichijoji Sagarimatsu-cho, Sakyo-ku,
Kyoto 606-8152, JAPAN}
\email{\href{mailto:quawai@me.com}{quawai@me.com} (Corresponding Author)}
\begin{abstract}
We report on the current status of our on-going project to develop
well-behaved hypersequent calculi for intermediate predicate logics,
such as the linearity axiom $\LIN\colon\left(\varphi\to\psi\right)\lor\left(\psi\to\varphi\right)$
and the constant domain axiom $\ACD\colon\forall x\left(\varphi\lor\psi\left(x\right)\right)\to\varphi\lor\forall x\psi\left(x\right)$.
\end{abstract}

\maketitle

\section{\label{sec:Introduction}Introduction}

Gentzen-style sequent calculus is a proof system for sequents $\Gamma\To\Delta$,
where $\Gamma$ and $\Delta$ are finite sequences of formulae. Since
the arrow symbol $\To$ behaves as meta-implication, implicational
axioms can be well transformed to inference rules. For example, the
$\land$-introduction axiom (schema) $\varphi\to\psi\to\varphi\land\psi$
can be reformulated as the following rule:
\[
\vcenter{\infer[{\rm (\land\mhyphen R)}]{\Gamma\To\Delta,\varphi\land\psi}{\Gamma\To\Delta,\varphi\amp\Gamma\To\Delta,\psi}}.
\]
On the other hand, sequent calculus does not well manipulate axioms
whose outermost logical symbols are not implications such as the linearity
axiom $\LIN\colon\left(\varphi\to\psi\right)\lor\left(\psi\to\varphi\right)$, the weak law of excluded middle $\WLEM\colon\neg\varphi\vee\neg\neg\varphi$, and the constant domain axiom $\ACD\colon\forall x\left(\varphi\lor\psi\left(x\right)\right)\to\varphi\lor\forall x\psi\left(x\right)$. Note that these axioms care about disjunctions (and universal quantifiers). See e.g. \citet{Kas07} for some fundamental
problems concerning $\ACD$.

Hypersequent calculus was first introduced by \citet{Avr87}. A \emph{hypersequent}
is a finite sequence $\left(\Gamma_{i}\To\varphi_{i}\right)_{i=1}^{n}$
of sequents, and is usually denoted as follows:
\[
\Gamma_{1}\To\Delta_{1}\mid\Gamma_{2}\To\Delta_{2}\mid\cdots\mid\Gamma_{n}\To\Delta_{n}.
\]
The sequents $\Gamma_{i}\To\Delta_{i}$ are called \emph{components}
of the hypersequent. Throughout this paper, we denote hypersequents
by meta-symbols $G,H,\ldots$, sequents by $S,T,\ldots$, and formulae
by $\varphi,\psi,\ldots$; the concatenation of (possibly empty) hypersequents
$G$ and $H$ by $G\mid H$.

The hypersequent calculus $\HLK$ of classical propositional logic
($\CL$) is given by the inference rules listed in \prettyref{tab:HLK-rules}.
\begin{table}
\centering%
\begin{tabular}{cc}
\noalign{\vskip2mm}
\multicolumn{2}{l}{Axioms}\tabularnewline[2mm]
\noalign{\vskip2mm}
$\infer[\left(\mathrm{Id}\right)]{\varphi\To\varphi}{}$ & $\infer[\left(\mathrm{Bot}\right)]{\bot\To\varphi}{}$\tabularnewline[2mm]
\noalign{\vskip2mm}
 & \tabularnewline[2mm]
\noalign{\vskip2mm}
\multicolumn{2}{l}{External structural rules}\tabularnewline[2mm]
\noalign{\vskip2mm}
\multicolumn{2}{c}{$\infer[\left(\mathrm{ew}\right)]{S\mid G}{G}$}\tabularnewline[2mm]
\noalign{\vskip2mm}
\multicolumn{2}{c}{$\infer[\left(\mathrm{ec}\right)]{S\mid G}{S\mid S\mid G}$}\tabularnewline[2mm]
\noalign{\vskip2mm}
\multicolumn{2}{c}{$\infer[\left(\mathrm{ee}\right)]{G\mid T\mid S\mid H}{G\mid S\mid T\mid H}$}\tabularnewline[2mm]
\noalign{\vskip2mm}
 & \tabularnewline[2mm]
\noalign{\vskip2mm}
\multicolumn{2}{l}{Internal structural rules}\tabularnewline[2mm]
\noalign{\vskip2mm}
$\infer[\left(\mathrm{iw\mhyphen L}\right)]{\varphi,\Gamma\To\Delta\mid G}{\Gamma\To\Delta\mid G}$ & $\infer[\left(\mathrm{iw\mhyphen R}\right)]{\Gamma\To\Delta,\psi\mid G}{\Gamma\To\Delta\mid G}$\tabularnewline[2mm]
\noalign{\vskip2mm}
$\infer[\left(\mathrm{ic\mhyphen L}\right)]{\varphi,\Gamma\To\Delta\mid G}{\varphi,\varphi,\Gamma\To\Delta\mid G}$ & $\infer[\left(\mathrm{ic\mhyphen R}\right)]{\Gamma\To\Delta,\psi\mid G}{\Gamma\To\Delta,\psi,\psi\mid G}$\tabularnewline[2mm]
\noalign{\vskip2mm}
$\infer[\left(\mathrm{ie\mhyphen L}\right)]{\Gamma_{1},\psi,\psi,\Gamma_{2}\To\Delta\mid G}{\Gamma_{1},\varphi,\psi,\Gamma_{2}\To\Delta\mid G}$ & $\infer[\left(\mathrm{ie\mhyphen R}\right)]{\Gamma\To\Delta_{1},\psi,\varphi,\Delta_{2}\mid G}{\Gamma\To\Delta_{1},\varphi,\psi,\Delta_{2}\mid G}$\tabularnewline[2mm]
\noalign{\vskip2mm}
 & \tabularnewline[2mm]
\noalign{\vskip2mm}
\multicolumn{2}{l}{Cut}\tabularnewline[2mm]
\noalign{\vskip2mm}
\multicolumn{2}{c}{$\infer[\left(\mathrm{cut}\right)]{\Gamma_{0},\Gamma_{1}\To\Delta_{0},\Delta_{1}\mid G}{\Gamma_{0}\To\Delta_{0},\delta\mid G\amp\delta,\Gamma_{1}\To\Delta_{1}\mid G}$}\tabularnewline[2mm]
\noalign{\vskip2mm}
 & \tabularnewline[2mm]
\noalign{\vskip2mm}
\multicolumn{2}{l}{Logical rules}\tabularnewline[2mm]
\noalign{\vskip2mm}
$\infer[\left(\mathrm{\land_{i}\mhyphen L}\right)]{\varphi_{1}\land\varphi_{2}\To\Delta\mid G}{\varphi_{i},\Gamma\To\Delta\mid G}$ & $\infer[\left(\mathrm{\land\mhyphen R}\right)]{\Gamma\To\Delta,\varphi_{1}\land\varphi_{2}\mid G}{\Gamma\To\Delta,\varphi_{1}\mid G\amp\Gamma\To\Delta,\varphi_{2}\mid G}$\tabularnewline[2mm]
\noalign{\vskip2mm}
$\infer[\left(\mathrm{\lor\mhyphen L}\right)]{\varphi_{1}\lor\varphi_{2},\Gamma\To\Delta\mid G}{\varphi_{1},\Gamma\To\Delta\mid G\amp\varphi_{2},\Gamma\To\Delta\mid G}$ & $\infer[\left(\mathrm{\lor_{i}\mhyphen R}\right)]{\Gamma\To\Delta,\varphi_{1}\lor\varphi_{2}\mid G}{\Gamma\To\Delta,\varphi_{i}\mid G}$\tabularnewline[2mm]
\noalign{\vskip2mm}
$\infer[\left(\mathrm{\to\mhyphen L}\right)]{\varphi\to\psi,\Gamma\To\Delta\mid G}{\Gamma\To\Delta,\varphi\mid G\amp\psi,\Gamma\To\Delta\mid G}$ & $\infer[\left(\mathrm{\to\mhyphen R}\right)]{\Gamma\To\Delta,\varphi\to\psi\mid G}{\varphi,\Gamma\To\Delta,\psi\mid G}$\tabularnewline[2mm]
\end{tabular}\caption{\label{tab:HLK-rules}}
\end{table}
 
\begin{fact}
\begin{enumerate}
\item If $\HLK\vdash\Gamma_{1}\To\Delta_{1}\mid\cdots\mid\Gamma_{n}\To\Delta_{n}$,
then $\bigvee_{i=1}^{n}\left(\bigwedge\Gamma_{i}\to\bigvee\Delta_{i}\right)$
is $\CL$-valid.
\item If $\varphi$ is $\CL$-valid, then $\HLK\vdash\To\varphi$.
\end{enumerate}
\end{fact}

We can obtain the hypersequent calculus for intuitionistic propositional
logic ($\INT$) by mimicking Gentzen's $\LJ$ or Maehara's $\LJ'$.
More precisely, $\HLJ$ is the subsystem of $\HLK$, where sequents
are restricted to \emph{single-conclusion}; and $\HLJ'$ is the subsystem
of $\HLK$, where the rule $\left(\mathrm{\to\mhyphen R}\right)$
is restricted to
\[
\vcenter{\infer[\left(\mathrm{\to\mhyphen R'}\right)]{\Gamma\To\varphi\to\psi\mid G}{\varphi,\Gamma\To\psi\mid G}}.
\]

\begin{fact}
\begin{enumerate}
\item If $\HLJ\vdash\Gamma_{1}\To\varphi_{1}\mid\cdots\mid\Gamma_{n}\To\varphi_{n}$,
then $\bigvee_{i=1}^{n}\left(\bigwedge\Gamma_{i}\to\varphi_{i}\right)$
is $\INT$-valid.
\item If $\HLJ'\vdash\Gamma_{1}\To\Delta_{1}\mid\cdots\mid\Gamma_{n}\To\Delta_{n}$,
then $\bigvee_{i=1}^{n}\left(\bigwedge\Gamma_{i}\to\bigvee\Delta_{i}\right)$
is $\INT$-valid.
\item If $\varphi$ is $\INT$-valid, then $\HLJ\vdash\To\varphi$ and $\HLJ'\vdash\To\varphi$.
\end{enumerate}
\end{fact}

The pipe symbol $\mid$ can be interpreted as meta-disjunctions, so
hypersequent calculus well manipulates disjunctive axioms. For example,
G\"odel--Dummett propositional logic $\GD$ (i.e. $\INT+\LIN$)
can be characterised by the following structural rule, called the
\emph{communication rule}:

\[
\vcenter{\infer[\left(\com\right)]{\Gamma,\Delta'\To\Theta\mid\Gamma',\Delta\To\Theta'\mid G}{\Gamma,\Delta\To\Theta\mid G\amp\Gamma',\Delta'\To\Theta'\mid G}}.
\]
This rule is an intermediate between the external (hypersequent-level)
structure and the internal (sequent-level) structure.
\begin{fact}[\citet{Avr91}, \citep{Avr96}]
\label{cor:HLJ+com-is-GD}
\begin{enumerate}
\item If $\HLJ+\left(\com\right)\vdash\Gamma_{1}\To\varphi_{1}\mid\cdots\mid\Gamma_{n}\To\varphi_{n}$,
then $\bigvee_{i=1}^{n}\left(\bigwedge\Gamma_{i}\to\varphi_{i}\right)$
is $\GD$-valid.
\item If $\HLJ'+\left(\com\right)\vdash\Gamma_{1}\To\Delta_{1}\mid\cdots\mid\Gamma_{n}\To\Delta_{n}$,
then $\bigvee_{i=1}^{n}\left(\bigwedge\Gamma_{i}\to\bigvee\Delta_{i}\right)$
is $\GD$-valid.
\item If $\varphi$ is $\GD$-valid, then $\HLJ+\left(\com\right)\vdash\To\varphi$
and $\HLJ'+\left(\com\right)\vdash\To\varphi$.
\end{enumerate}
\end{fact}

\begin{proof}
We only recall the proof of $\HLJ+\left(\com\right)\vdash\To\LIN$.
\[
\infer[\left(\mathrm{ec}\right)]{\To\left(\varphi\to\psi\right)\lor\left(\psi\to\varphi\right)}{\infer=[\left(\mathrm{\lor\mhyphen R}\right),\left(\mathrm{ee}\right)]{\To\left(\varphi\to\psi\right)\lor\left(\psi\to\varphi\right)\mid\To\left(\varphi\to\psi\right)\lor\left(\psi\to\varphi\right)}{\infer=[\left(\mathrm{\to\mhyphen R'}\right),\left(\mathrm{ee}\right)]{\To\varphi\to\psi\mid\To\psi\to\varphi}{\infer[\left(\com\right)]{\varphi,\varnothing\To\psi\mid\psi,\varnothing\To\varphi}{\infer[\left(\mathrm{Id}\right)]{\varphi,\varnothing\To\varphi}{}\amp\infer[\left(\mathrm{Id}\right)]{\psi,\varnothing\To\psi}{}}}}}\qedhere
\]
\end{proof}
Let us move on to predicate logics. The hypersequent calculi for intuitionistic
predicate logic ($\forall\INT$) can be obtained by adding $\HLJ$
and $\HLJ'$ with the quantifier rules (\prettyref{tab:Quantifier-rules}).
We refer to the resulting systems as $\forall\HLJ$ and $\forall\HLJ'$,
respectively.
\begin{table}
\centering%
\begin{tabular}{cc}
\noalign{\vskip2mm}
\multicolumn{2}{l}{Quantifier rules}\tabularnewline[2mm]
\noalign{\vskip2mm}
$\infer[\left(\mathrm{\forall\mhyphen L}\right)]{\forall x\varphi,\Gamma\To\Delta\mid G}{\left[t/x\right]\varphi,\Gamma\To\Delta\mid G}$ & $\infer[\left(\mathrm{\forall\mhyphen R_{ss}}\right)]{\Gamma\To\forall x\varphi}{\Gamma\To\varphi}$\tabularnewline[2mm]
\noalign{\vskip2mm}
 & provided that $x$ does not freely occur in $\Gamma$.\tabularnewline[2mm]
\noalign{\vskip2mm}
$\infer[\left(\mathrm{\exists\mhyphen L_{s}}\right)]{\exists x\varphi,\Gamma\To\Delta}{\varphi,\Gamma\To\Delta}$ & $\infer[\left(\mathrm{\exists\mhyphen R}\right)]{\Gamma\To\Delta,\exists x\psi\mid G}{\Gamma\To\Delta,\left[t/x\right]\psi\mid G}$\tabularnewline[2mm]
\noalign{\vskip2mm}
provided that $x$ does not freely occur in $\Gamma,\Delta$. & \tabularnewline[2mm]
\end{tabular}\caption{\label{tab:Quantifier-rules}}
\end{table}
 As the eigenvariable condition suggests, an (open) hypersequent $\left(\Gamma_{1}\To\Delta_{1}\right)\left(\vec{x}\right)\mid\cdots\mid\left(\Gamma_{n}\To\Delta_{n}\right)\left(\vec{x}\right)$
with free variables $\vec{x}$ represents a closed formula $\forall\vec{x}\bigvee_{i=1}^{n}\left(\bigwedge\Gamma_{i}\left(\vec{x}\right)\to\bigvee\Delta_{i}\left(\vec{x}\right)\right)$.
\begin{fact}
\label{fact:HLJ1-is-INT1}
\begin{enumerate}
\item If $\forall\HLJ\vdash\Gamma_{1}\To\varphi_{1}\mid\cdots\mid\Gamma_{n}\To\varphi_{n}$,
the universal closure of $\bigvee_{i=1}^{n}\left(\bigwedge\Gamma_{i}\to\varphi_{i}\right)$
is $\forall\INT$-valid.
\item If $\forall\HLJ'\vdash\Gamma_{1}\To\Delta_{1}\mid\cdots\mid\Gamma_{n}\To\Delta_{n}$,
the universal closure of $\bigvee_{i=1}^{n}\left(\bigwedge\Gamma_{i}\to\bigvee\Delta_{i}\right)$
is $\forall\INT$-valid.
\item If $\varphi$ is $\forall\INT$-valid, then $\forall\HLJ\vdash\To\varphi$
and $\forall\HLJ'\vdash\To\varphi$.
\end{enumerate}
\end{fact}

We also consider the multi-component single-conclusioned $\forall$-right
rule:
\[
\vcenter{\infer[\left(\mathrm{\forall\mhyphen R_{ms}}\right)]{\Gamma\To\forall x\varphi\mid G}{\Gamma\To\varphi\mid G}},
\]
where the variable $x$ does not freely occur in the lower hypersequent.
Apparently the rule $\left(\mathrm{\forall\mhyphen R_{ms}}\right)$
asserts that $\forall\vec{w}\forall x\left(\left(\gamma\left(\vec{w}\right)\to\varphi\left(x,\vec{w}\right)\right)\vee\psi\left(\vec{w}\right)\right)$
implies $\forall\vec{w}\left(\left(\gamma\left(\vec{w}\right)\to\forall x\varphi\left(x,\left(\vec{w}\right)\right)\right)\vee\psi\left(\left(\vec{w}\right)\right)\right)$,
a form of $\ACD$. However, to extract $\ACD$ from $\left(\mathrm{\forall\mhyphen R_{ms}}\right)$,
we need the communication rule $\left(\com\right)$. The combination
of $\left(\mathrm{\forall\mhyphen R_{ms}}\right)$ and $\left(\com\right)$
characterises G\"odel--Dummett predicate logic $\forall\GD:=\forall\INT+\LIN+\ACD$.
\begin{fact}[\citet{BZ00}]
\label{fact:HLJ1+com-is-GD1}
\begin{enumerate}
\item If $\forall\HLJ+\left(\mathrm{\forall\mhyphen R_{ms}}\right)+\left(\com\right)\vdash\Gamma_{1}\To\varphi_{1}\mid\cdots\mid\Gamma_{n}\To\varphi_{n}$,
the universal closure of $\bigvee_{i=1}^{n}\left(\bigwedge\Gamma_{i}\to\varphi_{i}\right)$
is $\forall\GD$-valid.
\item If $\forall\HLJ'+\left(\mathrm{\forall\mhyphen R_{ms}}\right)+\left(\com\right)\vdash\Gamma_{1}\To\Delta_{1}\mid\cdots\mid\Gamma_{n}\To\Delta_{n}$,
the universal closure of $\bigvee_{i=1}^{n}\left(\bigwedge\Gamma_{i}\to\bigvee\Delta_{i}\right)$
is $\forall\GD$-valid.
\item If $\varphi$ is $\forall\GD$-valid, then $\forall\HLJ+\left(\mathrm{\forall\mhyphen R_{ms}}\right)+\left(\com\right)\vdash\To\varphi$
and $\forall\HLJ'+\left(\mathrm{\forall_{ms}\mhyphen R}\right)+\left(\com\right)\vdash\To\varphi$.
\end{enumerate}
\end{fact}

\begin{proof}
We only recall the proof of $\forall\HLJ+\left(\mathrm{\forall\mhyphen R_{ms}}\right)+\left(\com\right)\vdash\To\ACD$.
\[
\infer[\left(\mathrm{\to\mhyphen R'}\right)]{\To\forall x\left(\varphi\lor\psi\left(x\right)\right)\to\varphi\lor\forall x\psi\left(x\right)}{\infer[\left(\mathrm{ec}\right)]{\forall x\left(\varphi\lor\psi\left(x\right)\right)\To\varphi\lor\forall x\psi\left(x\right)}{\infer=[\left(\mathrm{\lor\mhyphen R}\right),\left(\mathrm{ee}\right)]{\forall x\left(\varphi\lor\psi\left(x\right)\right)\To\varphi\lor\forall x\psi\left(x\right)\mid\forall x\left(\varphi\lor\psi\left(x\right)\right)\To\varphi\lor\forall x\psi\left(x\right)}{\infer[\left(\mathrm{\forall\mhyphen R_{ms}}\right)]{\forall x\left(\varphi\lor\psi\left(x\right)\right)\To\forall x\psi\left(x\right)\mid\forall x\left(\varphi\lor\psi\left(x\right)\right)\To\varphi}{\infer=[\left(\mathrm{\forall\mhyphen L}\right),\left(\mathrm{ee}\right)]{\forall x\left(\varphi\lor\psi\left(x\right)\right)\To\psi\left(x\right)\mid\forall x\left(\varphi\lor\psi\left(x\right)\right)\To\varphi}{\infer[\left(\mathrm{\lor\mhyphen L}\right)]{\varphi\lor\psi\left(x\right)\To\psi\left(x\right)\mid\varphi\lor\psi\left(x\right)\To\varphi}{\infer[\left(\mathrm{ee}\right)]{\varphi\To\psi\left(x\right)\mid\varphi\lor\psi\left(x\right)\To\varphi}{\infer[\left(\mathrm{\lor\mhyphen L}\right)]{\varphi\lor\psi\left(x\right)\To\varphi\mid\varphi\To\psi\left(x\right)}{\infer=[\left(\mathrm{ew}\right)\left(\mathrm{ee}\right)]{\varphi\To\varphi\mid\varphi\To\psi\left(x\right)}{\infer[\left(\mathrm{Id}\right)]{\varphi\To\varphi}{}}\amp\infer[\left(\com\right)]{\psi\left(x\right)\To\varphi\mid\varphi\To\psi\left(x\right)}{\infer[\left(\mathrm{Id}\right)]{\psi\left(x\right)\To\psi\left(x\right)}{}\amp\infer[\left(\mathrm{Id}\right)]{\varphi\To\varphi}{}}}}\amp\infer=[\left(\mathrm{ew}\right)\left(\mathrm{ee}\right)]{\psi\left(x\right)\To\psi\left(x\right)\mid\varphi\lor\psi\left(x\right)\To\varphi}{\infer[\left(\mathrm{Id}\right)]{\psi\left(x\right)\To\psi\left(x\right)}{}}}}}}}}\qedhere
\]
\end{proof}
Similarly, the multi-component $\exists$-left rule
\[
\infer[\left(\mathrm{\exists\mhyphen L_{m}}\right)]{\exists x\varphi,\Gamma\To\Delta\mid G}{\varphi,\Gamma\To\Delta\mid G}
\]
asserts that $\forall\vec{w}\forall x\left(\left(\varphi\left(x,\vec{w}\right)\land\gamma\left(\vec{w}\right)\to\delta\left(\vec{w}\right)\right)\lor\psi\left(\vec{w}\right)\right)$
implies $\forall\vec{w}\left(\left(\exists x\varphi\left(x,\vec{w}\right)\land\gamma\left(\vec{w}\right)\to\delta\left(\vec{w}\right)\right)\lor\psi\left(\vec{w}\right)\right)$,
and depends on $\ACD$. Recall the (informal) proof of this assertion:
suppose $\forall x\left(\left(\varphi\left(x\right)\land\gamma\to\delta\right)\lor\psi\right)$.
Applying the axiom schema $\ACD$, we have $\forall x\left(\varphi\left(x\right)\land\gamma\to\delta\right)\lor\psi$.
Since $\forall x\left(\varphi\left(x\right)\land\gamma\to\delta\right)\to\left(\exists x\varphi\left(x\right)\land\gamma\to\delta\right)$
is $\forall\INT$-valid, we obtain the desired conclusion $\left(\exists x\varphi\left(x\right)\land\gamma\to\delta\right)\lor\psi$.
The system $\forall\HLJ'+\left(\mathrm{\exists\mhyphen L_{m}}\right)+\left(\com\right)$
is therefore sound with respect to $\forall\GD$.
\begin{problem}
$\forall\HLJ'+\left(\mathrm{\exists\mhyphen L_{m}}\right)+\left(\com\right)\vdash\To\ACD$?
\end{problem}

\begin{rem}
One can obtain the proof figure of $\forall\HLJ+\left(\mathrm{\exists\mhyphen L_{m}}\right)+\left(\com\right)\vdash\To\varphi\land\exists x\psi\left(x\right)\to\exists x\left(\varphi\land\psi\left(x\right)\right)$
as the dual of the proof figure in \prettyref{fact:HLJ1+com-is-GD1}:
\[
\infer[\left(\mathrm{\to\mhyphen R'}\right)]{\To\varphi\land\exists x\psi\left(x\right)\to\exists x\left(\varphi\land\psi\left(x\right)\right)}{\infer[\left(\mathrm{ec}\right)]{\varphi\land\exists x\psi\left(x\right)\To\exists x\left(\varphi\land\psi\left(x\right)\right)}{\infer=[\left(\mathrm{\land\mhyphen L}\right),\left(\mathrm{ee}\right)]{\varphi\land\exists x\psi\left(x\right)\To\exists x\left(\varphi\land\psi\left(x\right)\right)\mid\varphi\land\exists x\psi\left(x\right)\To\exists x\left(\varphi\land\psi\left(x\right)\right)}{\infer[\left(\mathrm{\exists\mhyphen L_{m}}\right)]{\exists x\psi\left(x\right)\To\exists x\left(\varphi\land\psi\left(x\right)\right)\mid\varphi\To\exists x\left(\varphi\land\psi\left(x\right)\right)}{\infer=[\left(\mathrm{\exists\mhyphen R}\right),\left(\mathrm{ee}\right)]{\psi\left(x\right)\To\exists x\left(\varphi\land\psi\left(x\right)\right)\mid\varphi\To\exists x\left(\varphi\land\psi\left(x\right)\right)}{\infer[\left(\mathrm{\land\mhyphen R}\right)]{\psi\left(x\right)\To\varphi\land\psi\left(x\right)\mid\varphi\To\varphi\land\psi\left(x\right)}{\infer[\left(\mathrm{ee}\right)]{\psi\left(x\right)\To\varphi\mid\varphi\To\varphi\land\psi\left(x\right)}{\infer[\left(\mathrm{\land\mhyphen R}\right)]{\varphi\To\varphi\land\psi\left(x\right)\mid\psi\left(x\right)\To\varphi}{\infer=[\left(\mathrm{ew}\right)\left(\mathrm{ee}\right)]{\varphi\To\varphi\mid\psi\left(x\right)\To\varphi}{\infer[\left(\mathrm{Id}\right)]{\varphi\To\varphi}{}}\amp\infer[\left(\com\right)]{\varphi\To\psi\left(x\right)\mid\psi\left(x\right)\To\varphi}{\infer[\left(\mathrm{Id}\right)]{\varphi\To\varphi}{}\amp\infer[\left(\mathrm{Id}\right)]{\psi\left(x\right)\To\psi\left(x\right)}{}}}}\amp\infer=[\left(\mathrm{ew}\right)\left(\mathrm{ee}\right)]{\psi\left(x\right)\To\psi\left(x\right)\mid\varphi\To\varphi\land\psi\left(x\right)}{\infer[\left(\mathrm{Id}\right)]{\psi\left(x\right)\To\psi\left(x\right)}{}}}}}}}}
\]
Evidently the meta-formula $\varphi\land\exists x\psi\left(x\right)\to\exists x\left(\varphi\land\psi\left(x\right)\right)$
is $\forall\INT$-valid.
\end{rem}

Hypersequent calculi for intermediate logics such as $\GD$, $\forall\GD$
and $\forall\INT+\LIN$ have been extensively studied. See e.g. \citep{Avr91,Avr96,BZ00,BLZ13,Cia05,Cia2013,CRW2014,Tiu11}.

We aims to develop well-behaved proof systems for $\forall\INT+\ACD$
and $\forall\INT+\LIN$ via hypersequent calculus. For our purpose,
it is beneficial to specify the sources of $\ACD$ and $\LIN$ in
$\forall\HLJ+\left(\mathrm{\forall_{ms}\mhyphen R}\right)+\left(\com\right)$.
In \prettyref{sec:Splitting}, we introduce the right split rule (rs)
and the left split rule $\left(\ls\right)$ to clarify the communication
rule $\left(\com\right)$. We prove that
\begin{enumerate}
\item $\HLJ'+\left(\rs\right)$ and $\HLJ'+\left(\ls\right)$ are sound
and complete with respect to $\GD$;
\item $\forall\HLJ'+\left(\mathrm{\forall\mhyphen R_{ms}}\right)+\left(\rs\right)$
are sound and complete with respect to $\forall\GD$; and
\item $\forall\HLJ'+\left(\mathrm{\forall\mhyphen R_{ms}}\right)+\left(\ls\right)$
is equivalent to or stronger than $\forall\INT+\LIN$.
\end{enumerate}
In \prettyref{sec:Forall-rules}, we show that
\begin{enumerate}
\item $\forall\HLJ+\left(\com\right)$ and $\forall\HLJ'+\left(\com\right)$
are sound and complete with respect to $\forall\INT+\LIN$; and
\item $\forall\HLJ+\left(\mathrm{\forall\mhyphen R_{ms}}\right)+\left(\mathrm{\exists\mhyphen L_{m}}\right)$
and $\forall\HLJ'+\left(\mathrm{\forall\mhyphen R_{ms}}\right)+\left(\mathrm{\exists\mhyphen L_{m}}\right)$
are sound and complete with respect to $\forall\INT$.
\end{enumerate}
In \prettyref{sec:Future-work}, we conclude the paper with some future
research directions.

\section{\label{sec:Splitting}Splitting rules}

We analyse the communication rule by dividing it into two rules. We
first consider the \emph{right split rule}:
\[
\vcenter{\infer[\left(\rs\right)]{\Gamma\To\Delta_{1}\mid\Gamma\To\Delta_{2}\mid G}{\Gamma\To\Delta_{1},\Delta_{2}\mid G}}.
\]
In the algebraic point of view, this rule corresponds to the inequality
$\left(\gamma\to\delta_{1}\vee\delta_{2}\right)\leq\left(\gamma\to\delta_{1}\right)\vee\left(\gamma\to\delta_{2}\right)$,
which is known to be equivalent to $\LIN$ (see \citet[Proposition 2]{DMKJ20}).
\begin{thm}
\begin{enumerate}
\item $\HLJ'+\left(\rs\right)$ proves $\LIN$.
\item $\forall\HLJ'+\left(\mathrm{\forall\mhyphen R_{ms}}\right)+\left(\rs\right)$
proves $\ACD$.
\end{enumerate}
\end{thm}

\begin{proof}
The linearity axiom:
\[
\infer[\left(\mathrm{ec}\right)]{\To\left(\varphi\to\psi\right)\vee\left(\psi\to\varphi\right)}{\infer=[\left(\mathrm{\to\mhyphen R'}\right),\left(\mathrm{\lor\mhyphen R}\right),\left(\mathrm{ee}\right)]{\To\left(\varphi\to\psi\right)\lor\left(\psi\to\varphi\right)\mid\To\left(\varphi\to\psi\right)\lor\left(\psi\to\varphi\right)}{\infer[\left(\mathrm{cut}\right)]{\varphi\To\psi\mid\psi\To\varphi}{\infer[\left(\mathrm{\lor_{1}\mhyphen R}\right)]{\varphi\To\varphi\lor\psi}{\infer[\left(\mathrm{Id}\right)]{\varphi\To\varphi}{}}\amp\infer[\left(\mathrm{ee}\right)]{\varphi\lor\psi\To\psi\mid\psi\To\varphi}{\infer[\left(\mathrm{cut}\right)]{\psi\To\varphi\mid\varphi\lor\psi\To\psi}{\infer[\left(\mathrm{\lor_{2}\mhyphen R}\right)]{\psi\To\varphi\lor\psi}{\infer[\left(\mathrm{Id}\right)]{\psi\To\psi}{}}\amp\infer[\left(\rs\right)]{\varphi\lor\psi\To\varphi\mid\varphi\lor\psi\To\psi}{\infer[\left(\mathrm{\lor\mhyphen L}\right)]{\varphi\lor\psi\To\varphi,\psi}{\infer[\left(\mathrm{iw\mhyphen R}\right)]{\varphi\To\varphi,\psi}{\infer[\left(\mathrm{Id}\right)]{\varphi\To\varphi}{}}\amp\infer=[\left(\mathrm{iw\mhyphen R}\right)\left(\mathrm{ie\mhyphen R}\right)]{\psi\To\varphi,\psi}{\infer[\left(\mathrm{Id}\right)]{\psi\To\psi}{}}}}}}}}}
\]
The constant domain axiom:
\[
\infer[\left(\mathrm{\to\mhyphen R'}\right)]{\To\forall x\left(\varphi\lor\psi\left(x\right)\right)\to\varphi\lor\forall x\psi\left(x\right)}{\infer[\left(\mathrm{ec}\right)]{\forall x\left(\varphi\lor\psi\left(x\right)\right)\To\varphi\lor\forall x\psi\left(x\right)}{\infer=[\left(\mathrm{\lor\mhyphen R}\right),\left(\mathrm{ee}\right)]{\forall x\left(\varphi\lor\psi\left(x\right)\right)\To\varphi\lor\forall x\psi\left(x\right)\mid\forall x\left(\varphi\lor\psi\left(x\right)\right)\To\varphi\lor\forall x\psi\left(x\right)}{\infer[\left(\mathrm{\forall\mhyphen R_{ms}}\right)]{\forall x\left(\varphi\lor\psi\left(x\right)\right)\To\varphi\mid\forall x\left(\varphi\lor\psi\left(x\right)\right)\To\forall x\psi\left(x\right)}{\infer[\left(\rs\right)]{\forall x\left(\varphi\lor\psi\left(x\right)\right)\To\varphi\mid\forall x\left(\varphi\lor\psi\left(x\right)\right)\To\psi\left(x\right)}{\infer[\left(\mathrm{\forall\mhyphen L}\right)]{\forall x\left(\varphi\lor\psi\left(x\right)\right)\To\varphi,\psi\left(x\right)}{\infer[\left(\mathrm{\lor\mhyphen L}\right)]{\varphi\lor\psi\left(x\right)\To\varphi,\psi\left(x\right)}{\infer[\left(\mathrm{iw\mhyphen R}\right)]{\varphi\To\varphi,\psi\left(x\right)}{\infer[\left(\mathrm{Id}\right)]{\varphi\To\varphi}{}}\amp\infer=[\left(\mathrm{iw\mhyphen R}\right)\left(\mathrm{ie\mhyphen R}\right)]{\psi\left(x\right)\To\varphi,\psi\left(x\right)}{\infer[\left(\mathrm{Id}\right)]{\psi\left(x\right)\To\psi\left(x\right)}{}}}}}}}}}\qedhere
\]
\end{proof}
\begin{cor}[Completeness]
\begin{enumerate}
\item If $\varphi$ is $\GD$-valid, then $\HLJ'+\left(\rs\right)\vdash\To\varphi$.
\item If $\varphi$ is $\forall\GD$-valid, then $\forall\HLJ'+\left(\mathrm{\forall\mhyphen R_{ms}}\right)+\left(\rs\right)\vdash\To\varphi$.
\end{enumerate}
\end{cor}

\begin{thm}[Soundness]
\begin{enumerate}
\item If $\HLJ'+\left(\rs\right)\vdash\Gamma_{1}\To\Delta_{1}\mid\cdots\mid\Gamma_{n}\To\Delta_{n}$,
then $\bigvee_{i=1}^{n}\left(\bigwedge\Gamma_{i}\to\bigvee\Delta_{i}\right)$
is $\GD$-valid.
\item If $\forall\HLJ'+\left(\mathrm{\forall\mhyphen R_{ms}}\right)+\left(\rs\right)\vdash\Gamma_{1}\To\Delta_{1}\mid\cdots\mid\Gamma_{n}\To\Delta_{n}$,
then the universal closure of $\bigvee_{i=1}^{n}\left(\bigwedge\Gamma_{i}\to\bigvee\Delta_{i}\right)$
is $\forall\GD$-valid.
\end{enumerate}
\end{thm}

\begin{proof}
It suffices to show that $\left(\forall\right)\NJ+\LIN\vdash\left(\gamma\to\delta_{1}\vee\delta_{2}\right)\to\left(\gamma\to\delta_{1}\right)\vee\left(\gamma\to\delta_{2}\right)$.\footnotesize
\[
\infer[9]{\left(\gamma\to\delta_{1}\lor\delta_{2}\right)\to\left(\gamma\to\delta_{1}\right)\lor\left(\gamma\to\delta_{2}\right)}{\infer[7,8]{\left(\gamma\to\delta_{1}\right)\lor\left(\gamma\to\delta_{2}\right)}{\infer{\left(\gamma\to\delta_{1}\right)\lor\left(\gamma\to\delta_{2}\right)}{\infer[3]{\gamma\to\delta_{2}}{\infer[1,2]{\delta_{1}}{\infer[1]{\delta_{2}}{}\amp\infer{\delta_{2}}{\infer[7]{\delta_{1}\to\delta_{2}}{}\amp\infer[2]{\delta_{1}}{}}\amp\infer{\delta_{1}\lor\delta_{2}}{\infer[9]{\gamma\to\delta_{1}\lor\delta_{2}}{}\amp\infer[3]{\gamma}{}}}}}\amp\infer{\left(\gamma\to\delta_{1}\right)\lor\left(\gamma\to\delta_{2}\right)}{\infer[6]{\gamma\to\delta_{1}}{\infer[4,5]{\delta_{1}}{\infer[4]{\delta_{1}}{}\amp\infer{\delta_{1}}{\infer[8]{\delta_{2}\to\delta_{1}}{}\amp\infer[5]{\delta_{2}}{}}\amp\infer{\delta_{1}\lor\delta_{2}}{\infer[9]{\gamma\to\delta_{1}\lor\delta_{2}}{}\amp\infer[6]{\gamma}{}}}}}\amp\infer[\left(\LIN\right)]{\left(\delta_{1}\to\delta_{2}\right)\lor\left(\delta_{2}\to\delta_{1}\right)}{}}}\normalsize\qedhere
\]
 
\end{proof}
As the dual form of the right split, one can consider the \emph{left
split rule}:
\[
\vcenter{\infer[\left(\ls\right)]{\Gamma_{1}\To\Delta\mid\Gamma_{2}\To\Delta\mid G}{\Gamma_{1},\Gamma_{2}\To\Delta\mid G}}.
\]
This rule corresponds to the inequality $\left(\gamma_{1}\land\gamma_{2}\right)\to\delta\leq\left(\gamma_{1}\to\delta\right)\lor\left(\gamma_{2}\to\delta\right)$,
which is equivalent to $\LIN$ (see \citet[Proposition 2]{DMKJ20}).
\begin{thm}
$\HLJ'+\left(\ls\right)$ proves the generalised De Morgan's law $\GDM\colon\left(\left(\gamma_{1}\land\gamma_{2}\right)\to\delta\right)\to\left(\gamma_{1}\to\delta\right)\lor\left(\gamma_{2}\to\delta\right)$.
\end{thm}

\begin{proof}
\[
\infer[\left(\mathrm{\to\mhyphen R'}\right)]{\To\left(\gamma_{1}\wedge\gamma_{2}\to\delta\right)\to\left(\gamma_{1}\to\delta\right)\vee\left(\gamma_{2}\to\delta\right)}{\infer[\left(\mathrm{ec}\right)]{\gamma_{1}\wedge\gamma_{2}\to\delta\To\left(\gamma_{1}\to\delta\right)\vee\left(\gamma_{2}\to\delta\right)}{\infer=[\left(\mathrm{\lor\mhyphen R}\right),\left(\mathrm{ee}\right)]{\gamma_{1}\wedge\gamma_{2}\to\delta\To\left(\gamma_{1}\to\delta\right)\vee\left(\gamma_{2}\to\delta\right)\mid\gamma_{1}\wedge\gamma_{2}\to\delta\To\left(\gamma_{1}\to\delta\right)\vee\left(\gamma_{2}\to\delta\right)}{\infer=[\left(\mathrm{\to\mhyphen R'}\right),\left(\mathrm{ee}\right)]{\gamma_{1}\wedge\gamma_{2}\to\delta\To\gamma_{1}\to\delta\mid\gamma_{1}\wedge\gamma_{2}\to\delta\To\gamma_{2}\to\delta}{\infer[\left(\ls\right)]{\gamma_{1},\gamma_{1}\wedge\gamma_{2}\to\delta\To\delta\mid\gamma_{2},\gamma_{1}\wedge\gamma_{2}\to\delta\To\delta}{\infer=[\left(\mathrm{iw\mhyphen L}\right)\left(\mathrm{ie\mhyphen L}\right)]{\gamma_{1},\gamma_{1}\wedge\gamma_{2}\to\delta,\gamma_{2},\gamma_{1}\wedge\gamma_{2}\to\delta\To\delta}{\infer[\left(\mathrm{\to\mhyphen L}\right)]{\gamma_{1}\wedge\gamma_{2}\to\delta,\gamma_{1},\gamma_{2}\To\delta}{\infer[\left(\mathrm{\land\mhyphen R}\right)]{\gamma_{1},\gamma_{2}\To\gamma_{1}\wedge\gamma_{2}}{\infer=[\left(\mathrm{iw\mhyphen L}\right)\left(\mathrm{ie\mhyphen L}\right)]{\gamma_{1},\gamma_{2}\To\gamma_{1}}{\infer[\left(\mathrm{Id}\right)]{\gamma_{1}\To\gamma_{1}}{}}\amp\infer[\left(\mathrm{iw\mhyphen L}\right)]{\gamma_{1},\gamma_{2}\To\gamma_{2}}{\infer[\left(\mathrm{Id}\right)]{\gamma_{2}\To\gamma_{2}}{}}}\amp\infer=[\left(\mathrm{iw\mhyphen L}\right)\left(\mathrm{ie\mhyphen L}\right)]{\delta,\gamma_{1},\gamma_{2}\To\delta}{\infer[\left(\mathrm{Id}\right)]{\delta\To\delta}{}}}}}}}}}\qedhere
\]
\end{proof}
\begin{cor}[Completeness]
\begin{enumerate}
\item If $\varphi$ is $\GD$-valid, then $\HLJ'+\left(\ls\right)\vdash\To\varphi$.
\item If $\varphi$ is $\forall\INT+\LIN$-valid, then $\forall\HLJ'+\left(\mathrm{\forall\mhyphen R_{ms}}\right)+\left(\ls\right)\vdash\To\varphi$.
\end{enumerate}
\end{cor}

\begin{proof}
Trivial. Note that $\GDM$ implies $\LIN$ in $\left(\forall\right)\INT$.
\end{proof}
\begin{thm}[Soundness]
\begin{enumerate}
\item If $\HLJ'+\left(\ls\right)\vdash\Gamma_{1}\To\Delta_{1}\mid\cdots\mid\Gamma_{n}\To\Delta_{n}$,
then $\bigvee_{i=1}^{n}\left(\bigwedge\Gamma_{i}\to\bigvee\Delta_{i}\right)$
is $\GD$-valid.
\item If $\forall\HLJ'+\left(\mathrm{\forall\mhyphen R_{ms}}\right)+\left(\ls\right)\vdash\Gamma_{1}\To\Delta_{1}\mid\cdots\mid\Gamma_{n}\To\Delta_{n}$,
the universal closure of $\bigvee_{i=1}^{n}\left(\bigwedge\Gamma_{i}\to\bigvee\Delta_{i}\right)$
is $\forall\GD$-valid.
\end{enumerate}
\end{thm}

\begin{proof}
Obvious from the well-known fact that $\left(\forall\right)\NJ+\LIN\vdash\GDM$.
\end{proof}
We have shown that $\HLJ'+\left(\ls\right)=\GD$ and $\forall\INT+\LIN\leq\forall\HLJ'+\left(\mathrm{\forall\mhyphen R_{ms}}\right)+\left(\ls\right)\leq\forall\GD$.
\begin{problem}
Decide the exact strength of $\forall\HLJ'+\left(\mathrm{\forall\mhyphen R_{ms}}\right)+\left(\ls\right)$.
\end{problem}

\section{\label{sec:Forall-rules}Restriction of $\forall$-right and $\exists$-left}

Recall that the proof of $\ACD$ in $\forall\HLJ+\left(\mathrm{\forall\mhyphen R_{ms}}\right)+\left(\com\right)$
(\prettyref{fact:HLJ1+com-is-GD1}) essentially uses the multi-component
single-conclusioned $\forall$-right rule $\left(\mathrm{\forall\mhyphen R_{ms}}\right)$.
\begin{thm}[Completeness]
If $\varphi$ is $\forall\INT+\LIN$-valid, then $\forall\HLJ+\left(\com\right)\vdash\To\varphi$
and $\forall\HLJ'+\left(\com\right)\vdash\To\varphi$.
\end{thm}

\begin{proof}
Obvious from $\HLJ+\left(\com\right)\vdash\LIN$ (\prettyref{cor:HLJ+com-is-GD})
and $\HLJ\subseteq\forall\HLJ\subseteq\forall\HLJ'$.
\end{proof}
\begin{thm}[Soundness]
\begin{enumerate}
\item If $\forall\HLJ+\left(\com\right)\vdash\Gamma_{1}\To\varphi_{1}\mid\cdots\mid\Gamma_{n}\To\varphi_{n}$,
the universal closure of $\bigvee_{i=1}^{n}\left(\bigwedge\Gamma_{i}\to\varphi_{i}\right)$
is $\forall\INT+\LIN$-valid.
\item If $\forall\HLJ'+\left(\com\right)\vdash\Gamma_{1}\To\Delta_{1}\mid\cdots\mid\Gamma_{n}\To\Delta_{n}$,
the universal closure of $\bigvee_{i=1}^{n}\left(\bigwedge\Gamma_{i}\to\bigvee\Delta_{i}\right)$
is $\forall\INT+\LIN$-valid.
\end{enumerate}
\end{thm}

\begin{proof}
One can verify that all the inference rules are $\forall\INT+\LIN$-valid.
\end{proof}
\begin{cor}
The rule $\left(\mathrm{\forall\mhyphen R_{ms}}\right)$ is not derived
from $\forall\HLJ$ or $\forall\HLJ'$.
\end{cor}

\begin{proof}
Otherwise, $\forall\HLJ\left(\forall\HLJ'\right)$ with $\left(\com\right)$
proves $\ACD$ by \prettyref{fact:HLJ1+com-is-GD1}, a contradiction.
\end{proof}
The principal source of $\ACD$ is the rule $\left(\mathrm{\forall\mhyphen R_{ms}}\right)$;
however, the rule $\left(\mathrm{\forall\mhyphen R_{ms}}\right)$
does not imply $\ACD$ solely.
\begin{thm}[Soundness]
\label{thm:strong-soundness}
\begin{enumerate}
\item If $\forall\HLJ+\left(\mathrm{\forall\mhyphen R_{ms}}\right)+\left(\mathrm{\exists\mhyphen L_{m}}\right)\vdash\Gamma_{1}\To\varphi_{1}\mid\cdots\mid\Gamma_{n}\To\varphi_{n}$,
the universal closure of $\bigwedge\Gamma_{i}\to\varphi_{i}$ is $\forall\INT$-valid
for some $i$.
\item If $\forall\HLJ'+\left(\mathrm{\forall\mhyphen R_{ms}}\right)+\left(\mathrm{\exists\mhyphen L_{m}}\right)\vdash\Gamma_{1}\To\Delta_{1}\mid\cdots\mid\Gamma_{n}\To\Delta_{n}$,
the universal closure of $\bigwedge\Gamma_{i}\to\bigvee\Delta_{i}$
is $\forall\INT$-valid for some $i$.
\end{enumerate}
\end{thm}

\begin{proof}
We only need to show that $\forall\HLJ'+\left(\mathrm{\forall\mhyphen R_{ms}}\right)\vdash S_{1}\mid\cdots\mid S_{n}$
implies $\forall\LJ'\vdash S_{i}$ for some $i$. Given a proof figure
$\Pi$ of $\forall\HLJ'+\left(\mathrm{\forall\mhyphen R_{ms}}\right)\vdash S_{1}\mid\cdots\mid S_{n}$,
we construct a proof figure $\Pi'$ of $\forall\LJ'\vdash S_{i}$
for some $i$ by induction on the structure of $\Pi$.
\begin{casenv}
\item If $\Pi$ is $\infer[\left(\mathrm{Id}\right)]{\varphi\To\varphi}{}$
or $\infer[\left(\mathrm{Bot}\right)]{\bot\To\varphi}{}$, then it
is a proof figure of $\forall\LJ'$ at the same time.
\item The last inference rule is one of the external structural rules. For
example, if the last rule is the external weakening rule
\[
\infer[\left(\mathrm{ew}\right)]{S_{1}\mid G}{\begin{array}{c}
\ddots\vdots\iddots\\
G
\end{array}}
\]
then we have constructed a proof figure of $\forall\LJ'\vdash S_{i}$
for some $S_{i}\in G$ by the induction hypothesis. The same applies
to external exchange and external contraction.
\item The last inference rule is the cut rule
\[
\infer[\left(\mathrm{cut}\right)]{\Gamma_{0},\Gamma_{1}\To\Delta_{0},\Delta_{1}\mid G}{\begin{array}{c}
\ddots\vdots\iddots\\
\Gamma_{0}\To\Delta_{0},\delta\mid G
\end{array}\amp\begin{array}{c}
\ddots\vdots\iddots\\
\delta,\Gamma_{1}\To\Delta_{1}\mid G
\end{array}}
\]
then one of the following cases holds by the induction hypothesis.
\begin{casenv}
\item There exists a proof figure of $\forall\LJ'\vdash S$ for some $S\in G$
as desired.
\item There exist proof figures $\Sigma_{0}$ and $\Sigma_{1}$ of $\forall\LJ'\vdash\Gamma_{0}\To\Delta_{0},\delta$
and $\forall\LJ'\vdash\delta,\Gamma_{1}\To\Delta_{1}$, respectively.
The desired proof figure of $\forall\LJ'\vdash\Gamma_{0},\Gamma_{1}\To\Delta_{0},\Delta_{1}$
is obtained as follows:
\[
\vcenter{\infer[\left(\mathrm{cut}\right)]{\Gamma_{0},\Gamma_{1}\To\Delta_{0},\Delta_{1}}{\Sigma_{0}\amp\Sigma_{1}}}.
\]
\end{casenv}
\item The last inference rule is one of the internal structural rules, the
logical rules and the quantifier rules. The same argument works well.
For example, if the last rule is the quantifier rule
\[
\infer[\left(\mathrm{\forall\mhyphen R_{ms}}\right)]{\Gamma\To\forall x\varphi\mid G}{\begin{array}{c}
\ddots\vdots\iddots\\
\Gamma\To\varphi\mid G
\end{array}}
\]
then one of the following cases holds by the induction hypothesis.
\begin{casenv}
\item There exists a proof figure of $\forall\LJ'\vdash S$ for some $S\in G$.
\item There exists a proof figure $\Sigma$ of $\forall\LJ'\vdash\Gamma\To\varphi$.
We obtain the desired proof figure of $\forall\LJ'\vdash\Gamma\To\forall x\varphi$:
\[
\vcenter{\infer[\left(\mathrm{\forall\mhyphen R}\right)]{\Gamma\To\forall x\varphi}{\Sigma}}.
\]
\end{casenv}
\end{casenv}
Note that this procedure does not increase the complexity of the proofs
(such as the number of symbols, formulae and steps).

\end{proof}
\begin{rem}
The hypersequent calculi $\HLK$, $\HLJ$, $\HLJ'$ and their predicate
versions have the strong soundness property in the sense of \prettyref{thm:strong-soundness}.
On the other hand, the hypersequent calculi with the communication
rule (or its variations such as the right split rule) does not possess
the strong soundness property. For example, the hypersequent $\varphi\To\psi\mid\psi\To\varphi$
is provable in such a system, but is neither $\varphi\To\psi$ nor
$\psi\To\varphi$.
\end{rem}

\section{\label{sec:Future-work}Future work}

A hypersequent calculus of $\forall\INT+\ACD$ can be obtained by
adding either the multi-component multi-conclusioned $\forall$-right
rule
\[
\vcenter{\infer[\left(\mathrm{\forall\mhyphen R_{mm}}\right)]{\Gamma\To\Delta,\forall x\varphi\mid G}{\Gamma\To\Delta,\varphi\mid G}}
\]
or the single-component multi-conclusioned $\forall$-right rule
\[
\vcenter{\infer[\left(\mathrm{\forall\mhyphen R_{sm}}\right)]{\Gamma\To\Delta,\forall x\varphi}{\Gamma\To\Delta,\varphi}}.
\]
This however makes no progress on proof theory of $\forall\INT+\ACD$
beyond the Gentzen-style proof system. In fact, Maehara's $\forall\LJ'$
with $\left(\mathrm{\forall\mhyphen R_{sm}}\right)$ gives a sequent
calculus for $\forall\INT+\ACD$ (see e.g. \citet{KS94}). The system
$\forall\HLJ'+\left(\mathrm{\forall\mhyphen R_{mm}}\right)+\left(\mathrm{\exists\mhyphen L_{m}}\right)$
is merely a hypersequent version of $\forall\LJ'+\left(\mathrm{\forall\mhyphen R_{sm}}\right)$.
\begin{problem}
Find a (well-behaved) hypersequent calculus for $\forall\INT+\ACD$,
where $\ACD$ is formulated as a \emph{structural rule}. Establish
the cut-elimination theorem and the Craig interpolation theorem for
such a system.
\end{problem}

A hypersequent $\left(\Gamma_{1}\To\Delta_{1}\right)\left(\vec{x}\right)\mid\cdots\mid\left(\Gamma_{n}\To\Delta_{n}\right)\left(\vec{x}\right)$
with free variables $\vec{x}$ can be translated to closed formulae
in two different ways:
\[
\forall\vec{x}\bigvee_{i=1}^{n}\left(\bigwedge\Gamma_{i}\left(\vec{x}\right)\to\bigvee\Delta_{i}\left(\vec{x}\right)\right),\quad\bigvee_{i=1}^{n}\forall\vec{x}\left(\bigwedge\Gamma_{i}\left(\vec{x}\right)\to\bigvee\Delta_{i}\left(\vec{x}\right)\right).
\]
In the first case, the free variables are considered to be \emph{shared}
with all components. In the second case, the free variables are considered
not to be shared. In order to manipulate these two translations explicitly,
one can introduce two kinds of variables, \emph{global variables}
and \emph{local variables}. We immediately observe that the global-to-local
conversion rule
\[
\infer{\To\varphi\mid\To\psi\left(x^{\mathrm{local}}\right)}{\To\varphi\mid\To\psi\left(x^{\mathrm{global}}\right)}
\]
corresponds to $\ACD\colon\forall x\left(\varphi\lor\psi\left(x\right)\right)\to\varphi\lor\forall x\psi\left(x\right)$.
It might be fruitful to investigate the sharing/unsharing rules (see
\prettyref{tab:Share-and-unshare}).
\begin{table}
\centering%
\begin{tabular}{c}
\noalign{\vskip2mm}
$\infer[\left(\mathrm{share}\right)]{\left[x^{\mathrm{global}}/x^{\mathrm{local}}\right]S\mid G}{S\mid G}$\tabularnewline[2mm]
\noalign{\vskip2mm}
where $x^{\mathrm{global}}$ does not freely occur in $S,G$.\tabularnewline[2mm]
\noalign{\vskip2mm}
\tabularnewline[2mm]
\noalign{\vskip2mm}
$\infer[\left(\mathrm{unshare}\right)]{\left[x^{\mathrm{local}}/x^{\mathrm{global}}\right]S\mid G}{S\mid G}$\tabularnewline[2mm]
\noalign{\vskip2mm}
where $x^{\mathrm{global}}$ does not freely occur in $G$ and $x^{\mathrm{local}}$
does not freely occur in $S$.\tabularnewline[2mm]
\end{tabular}\caption{\label{tab:Share-and-unshare}}
\end{table}

\begin{problem}
Develop hypersequent calculi with the distinction of global and local
variables.
\end{problem}

Hirai \citep{Hir12,Hir13} proposed hyper-lambda calculi, models of
concurrent computation. Simply typed hyper-lambda calculi correspond
to various propositional hypersequent calculi. Notably, the asynchronous
hyper-lambda calculus $\lambda\mathrm{\mhyphen GD}$ corresponds to
Avron's system $\HLJ+\left(\com\right)$ of $\GD$. Through the Curry--Howard
correspondence, we can shed light on the computational content of
the linearity axiom $\LIN$. Naturally, it is expected that the computational
content of $\ACD$ can be revealed by considering an appropriate dependently
typed hyper-lambda calculus.
\begin{problem}
Develop a dependently typed hyper-lambda calculus corresponding to
$\forall\INT+\ACD$.
\end{problem}

\subsection*{Author Contributions}

Conceptualisation and Methodology, T.I. (\prettyref{sec:Splitting}
and \ref{sec:Future-work}) and S.M. (\prettyref{sec:Forall-rules});
Investigation and Validation, T.I., S.M. and S.Q.; Writing---Original
Draft, T.I. and S.Q.; Writing---Review \& Editing, T.I., S.M. and
S.Q.; Visualisation, S.Q.; Project Administration, S.Q.

\bibliographystyle{IEEEtranSN}
\bibliography{ref}

\end{document}